\documentclass[11pt]{article}
\usepackage{amsmath, amscd, amssymb, latexsym, epsfig, color, amsthm,enumerate, mathtools}
\usepackage[all]{xy}
\usepackage{verbatim}
\numberwithin{equation}{section}

\newtheorem{theorem}{Theorem}[section]

\newtheorem{question}[theorem]{Question}
\newtheorem{corollary}[theorem]{Corollary}
\newtheorem{conjecture}[theorem]{Conjecture}
\newtheorem{lemma}[theorem]{Lemma}

\theoremstyle{definition}
\newtheorem{definition}[theorem]{Definition}
\newtheorem{example}[theorem]{Example}

\newtheorem{remark}[theorem]{Remark}

\DeclareMathOperator{\skel}{Skel}
\DeclareMathOperator\lk{\mathrm{lk}}
\DeclareMathOperator\antst{\mathrm{ast}}

\DeclareMathOperator\sn{\mathrm{sn}}

\DeclareMathOperator{\conv}{\mathrm{conv}}

\newcommand{\N}{{\mathbb N}}
\newcommand{\R}{{\mathbb R}}

\newcommand{\A}{{\mathcal A}}
\newcommand{\Qq}{{\mathcal Q}}

\newcommand{\I}{{\mathcal I}}

\newcommand{\Ss}{{\mathcal S}}
\newcommand{\T}{{\mathcal T}}
\newcommand{\F}{{\mathcal F}}
\newcommand{\PP}{{\mathcal P}}

\title{Many neighborly spheres}
\author{
	Isabella Novik\\
	\small Department of Mathematics\\[-0.8ex]
	\small University of Washington\\[-0.8ex]
	\small Seattle, WA 98195-4350, USA\\[-0.8ex]
	\small \texttt{novik@uw.edu}
	\and 
	Hailun Zheng \\
	\small Department of Mathematical Sciences\\[-0.8ex]
	\small University of Copenhagen\\[-0.8ex]
	\small Universitesparken 5, 2100 Copenhagen, Denmark \\[-0.8ex]
	\small \texttt{hz@math.ku.dk}
}
\begin{document}
\maketitle
\begin{abstract}
The result of Padrol \cite{Padrol-13} asserts that for every $d\geq 4$, there exist $2^{\Omega(n\log n)}$ distinct combinatorial types of $\lfloor d/2\rfloor$-neighborly simplicial $(d-1)$-spheres with $n$ vertices. We present a construction showing that for every $d\geq 5$, there are at least $2^{\Omega(n^{\lfloor (d-1)/2\rfloor})}$ such types.
\end{abstract}

\subsection*{Acknowledgements} Research of IN is partially 
		supported by NSF grants DMS-1664865 and DMS-1953815, and by Robert R.~\&  Elaine F.~Phelps Professorship in Mathematics. Research of HZ is partially supported by a postdoctoral fellowship from ERC grant 716424 - CASe. The authors are grateful to the referee for several clarifying questions.
		
\subsection*{MSC} 52B05,  52B70, 57Q15 

\section{Introduction}
A simplicial complex on $n$ vertices is $s$-neighborly if it has the same $(s-1)$-skeleton as the $(n-1)$-simplex on the same vertex set. 
Of special interest are  $\lfloor d/2\rfloor$-neighborly  $(d-1)$-spheres. They arise, for instance, in the context of Stanley's upper bound theorem \cite{Stanley75}. In this paper we address the question of how many  $\lfloor d/2\rfloor$-neighborly $(d-1)$-spheres with $n$ vertices there are.

This question is ultimately related to the questions of how many combinatorial types of (convex) simplicial $d$-polytopes with $n$ labeled vertices there are and how many combinatorial types of simplicial $(d-1)$-spheres with $n$ labeled vertices there are. Denote these numbers by $c(d,n)$ and $s(d,n)$, respectively. 
The asymptotic answer to the first question was given by Goodman and Pollack \cite{GoodmanPollack} followed by the work of Alon \cite{Alon-86}. They showed that there are very few polytopes: $c(d,n)=2^{\Theta(n\log n)}$ for $d\geq 4$.
In contrast to these results, Kalai \cite{Kal} proved that there is a very large number of simplicial spheres: for $d\geq 5$, $s(d,n)\geq 2^{\Omega(n^{\lfloor (d-1)/2\rfloor})}$. Furthermore, Pfeifle and Ziegler \cite{PfeiZieg} showed that $s(4,n)\geq 2^{\Omega(n^{5/4})}$.   The current record on the number of odd-dimensional simplicial spheres is due to Nevo, Santos, and Wilson \cite{NeSanWil} who established the following bound:  $s(2k,n)\geq 2^{\Omega(n^{k})}$  for all $k\geq 2$. In short, the best to-date lower bound for any $d\geq 4$ is $s(d,n)\geq 2^{\Omega(n^{\lfloor d/2\rfloor})}.$ On the other hand, Stanley's upper bound theorem implies that $s(d, n)\leq 2^{O(n^{\lfloor d/2\rfloor}\log n)}$ (see \cite[Section 4.2]{Kal}). This is the current best upper bound.

Despite the fact that most of spheres constructed in \cite{Kal, NeSanWil, PfeiZieg} are not neighborly, Kalai \cite[Section 6.3]{Kal} speculated that the number $\sn(d,n)$ of $\lfloor d/2\rfloor$-neighborly simplicial $(d-1)$-spheres with $n$ labeled vertices is very large and posited the following conjecture.
\begin{conjecture}For all $d\geq 4$,
	$$\lim_{n\to\infty} (\log \sn(d,n)/ \log s(d,n))=1.$$ 
\end{conjecture} 
Indeed, the currently best known lower bound on the number of $\lfloor d/2\rfloor$-neighborly $d$-polytopes with $n$ labeled vertices (due to Padrol \cite{Padrol-13}) is also the best known  lower bound on the total number of combinatorial types of $d$-polytopes with $n$ labeled vertices.  Padrol's paper built on and generalized  Shemer's sewing construction \cite{Shemer}, which was used to produce the previous record number of neighborly polytopes.  In addition to neighborly polytopes, Padrol was also able to construct a record number of non-realizable neighborly oriented matroids. Yet, Padrol's bounds only imply that $\sn(d,n) \geq 2^{\Omega(n\log n)}$.

While we are still very far from being able to shed light on Kalai's conjecture, we improve Padrol's bound and prove the following result.

\begin{theorem} \label{main-thm}
For all $d\geq 5$, $\sn(d,n) \geq 2^{\Omega(n^{\lfloor (d-1)/2\rfloor})}$.
\end{theorem}

Our construction utilizes Kalai's squeezed balls \cite{Kal}. In fact, the key to our proof is an observation that for certain choices of parameters, the difference of two squeezed $(2k-1)$-balls on $n$ vertices forms a $(k-1)$-neighborly and $(k-1)$-stacked $(2k-1)$-ball on the same vertex set, see Theorem~\ref{thm: prop of B_S}. These ``difference" balls are contained in the boundary complex of the cyclic $2k$-polytope on $n$ vertices, denoted as $\partial C_{2k}(n)$. They are extremely useful for our constructions in both even- and odd-dimensional cases. Indeed, on one hand, the boundary of such a ball $B$ is a $(k-1)$-neighborly $(2k-2)$-sphere on $n$ vertices. On the other hand, removing $B$ from $\partial C_{2k}(n)$ and patching the resulting hole with the cone over the boundary of $B$ produces a $k$-neighborly $(2k-1)$-sphere on $n+1$ vertices. 

A few historical remarks are in order. The first construction of polytopal $d$-balls with $n$ vertices that are both $r$-neighborly and $r$-stacked (for all parameters $r,d,n$ with $2\leq 2r\leq d$ and $n \geq d+1$) is due to McMullen and Walkup \cite{McMullenWalkup71}. The idea of finding inside a triangulated manifold $M$ a full-dimensional ball $B$ that is both $1$-neighborly (i.e., $B$ contains all vertices of $M$) and $1$-stacked, and using such balls to construct $2$-neighborly triangulations of manifolds was pioneered by Walkup \cite{Walkup}; for a much more recent use of the same idea see \cite[Section 5]{Swartz2009}.
 
The structure of the rest of the paper is as follows. In Section 2 we review basic definitions related to neighborliness, stackedness, and Kalai's squeezed balls. In Section 3 we describe our main construction, the relative squeezed balls. Finally, in Section 4 we prove Theorem \ref{main-thm}.

\section{Preliminaries} 
\subsection{Simplicial complexes} We start by providing a quick overview of the main objects of this paper --- simplicial complexes.
A \emph{simplicial complex} $\Delta$ with vertex set $V(\Delta)$ is a collection of subsets of $V(\Delta)$ that is closed under inclusion and contains all singletons: $\{v\}\in\Delta$ for all $v\in V(\Delta)$. The elements of $\Delta$ are called {\em faces}. The \emph{dimension of a face} $\tau\in\Delta$ is $\dim\tau:=|\tau|-1$. The \emph{dimension of $\Delta$}, $\dim\Delta$, is the maximum dimension of its faces. A face of a simplicial complex $\Delta$ is a \textit{facet} if it is maximal w.r.t.~inclusion. We say that $\Delta$ is \emph{pure} if all facets of $\Delta$ have the same dimension. We distinguish between the {\em empty complex} $\Delta=\{\emptyset\}$ whose only face is the empty set and the {\em void complex} $\Delta=\emptyset$ that has no faces (not even the empty set).

Let $\Delta$ be a simplicial complex. The \emph{$k$-skeleton} of $\Delta$, $\skel_k(\Delta)$, is the subcomplex of $\Delta$ consisting of all faces of dimension $\leq k$. If $\tau$ is a face of $\Delta$, then the {\em antistar of $\tau$} and 
the  {\em link of $\tau$ in $\Delta$} are the following subcomplexes of $\Delta$:
\[\antst(\tau, \Delta)=\{\sigma\in \Delta: \sigma \not\supseteq \tau\}, \quad 
\lk(\tau,\Delta):= \{\sigma\in \Delta \ : \ \sigma\cap\tau=\emptyset \mbox{ and } \sigma\cup\tau\in\Delta\}.
\]
If $\Delta$ is a pure simplicial complex and $\Gamma$ is a full-dimensional pure subcomplex of $\Delta$, then $\Delta\backslash \Gamma$ is the subcomplex of $\Delta$ generated by those facets of $\Delta$ that are not in $\Gamma$. Finally, if $\Delta$ and $\Gamma$ are simplicial complexes on disjoint vertex sets, then the \textit{join} of $\Delta$ and $\Gamma$ is the simplicial complex $\Delta*\Gamma = \{\sigma \cup \tau \ : \ \sigma \in \Delta \text{ and } \tau \in \Gamma\}$. In particular, the join of $\Delta$ with the empty complex is $\Delta$ while the join of $\Delta$ with the void complex is the void complex.

Let $V$ be a set of size $d+1$. Denote by $\overline{V}$ the $d$-dimensional simplex on $V$. Its boundary complex is $\partial\overline{V}:=\{\tau : \tau\subsetneq V\}$. Most of complexes considered in this paper are PL balls or PL spheres. (Here PL stands for piecewise linear.) A \emph{PL $d$-ball} is a simplicial complex PL homeomorphic to $\overline{V}$. Similarly, a \emph{PL $(d-1)$-sphere} is a simplicial complex PL homeomorphic to $\partial\overline{V}$. If $\Delta$ is a PL $d$-sphere and $\Gamma\subset\Delta$ is a PL $d$-ball, then so is $\Delta\backslash \Gamma$, see \cite{Hudson}. Furthermore, the link of any face in a PL sphere is a PL sphere. On the other hand, the link of a face $\tau$ in a PL $d$-ball $B$ is either a PL ball or a PL sphere; in the former case we say that $\tau$ is  a \emph{boundary face} of $B$, and in the latter case that $\tau$ is an \emph{interior face} of $B$. The \emph{boundary complex} of $B$, $\partial B$, is the subcomplex of $B$ that consists of all boundary faces of $B$; in particular, $\partial B$ is a PL $(d-1)$-sphere.

For a $(d-1)$-dimensional complex $\Delta$ and $-1\leq i\leq d-1$, we let $f_i = f_i(\Delta)$ be the number of $i$-dimensional faces of $\Delta$. The vector $f(\Delta)=(f_{-1},f_0,\ldots,f_{d-1})$ is called the {\em $f$-vector} of $\Delta$. We also define the \textit{$h$-vector} of $\Delta$, $h(\Delta)=(h_0,\ldots, h_d)$, by the following relation: $\sum_{j=0}^{d}h_j\lambda^{d-j}=\sum_{i=0}^{d}f_{i-1}(\lambda-1)^{d-i}$. In particular,  $f_{-1}=h_0=1$ and $f_{d-1}=\sum_{j=0}^{d}h_j$.

\subsection{Cyclic polytopes, neighborliness, and stackedness}
Let $i\geq 1$. We say that a simplicial complex $\Delta$ is \emph{$i$-neighborly} w.r.t.~$V$ (or simply $i$-neighborly) if $\skel_{i-1}(\Delta)=\skel_{i-1}(\overline{V})$.

Let $m: \R^d \to \R$, $t \mapsto (t, t^2, \dots, t^d)$, be the \emph{moment curve} in $\R^d$, and let $t_1<t_2<\dots <t_n$ be distinct real numbers, where $n>d$. The \emph{cyclic $d$-polytope} $C_d(n)$ is defined as the convex hull $\conv(m(t_1), \dots, m(t_n))$.  It is known that $C_d(n)$ is a simplicial $d$-polytope with $n$ vertices, that it is $\lfloor d/2\rfloor$-neighborly and that its combinatorial type is independent of the choice of $t_1,\ldots,t_n$. In the rest of the paper we treat the boundary complex of $C_d(n)$, $\partial C_d(n)$, as an abstract simplicial complex. In particular, we identify a vertex $m(t_i)$ with $i\in [n]:=\{1,2,\ldots, n\}$ and the vertex set of $\partial C_d(n)$ with $[n]$. The facets of $C_d(n)$ have a particularly nice description known as Gale's evenness condition \cite{Gale63}: 

\begin{lemma}
	Let $n>d \geq 2$, and  let $C_d(n)$ be the cyclic $d$-polytope. A $d$-subset $F\subset [n]$ forms a facet of $\partial C_d(n)$ if and only if for every $i<j$ not in $F$, the number of elements $\ell\in F$ between $i$ and $j$ is even. In particular, for $d=2k$, $F=\{i_1, i_1+1, i_2, i_2+1, \dots, i_k, i_k+1\}$ is a facet of $\partial C_{2k}(n)$ if $1\leq i_1$, $i_k\leq n-1$, and $i_{j}\leq i_{j+1}-2$ for all $1\leq j\leq k-1$.
\end{lemma}

A property that seems to be inseparable from neighborliness is that of stackedness. A PL $d$-ball $B$ is called \emph{$i$-stacked} (for some $0\leq i\leq d$), if all interior faces of $B$ are of dimension $\geq d-i$, that is, $\skel_{d-i-1}(B)=\skel_{d-i-1}(\partial B)$. In particular, $0$-stacked balls are simplices; $1$-stacked balls are also known in the literature as stacked balls. We will rely on the following basic properties, see for instance \cite[Lemma 2.2]{N-Z}. Part 2 of the statement below is stronger than the one provided in \cite[Lemma 2.2(2)]{N-Z}, but the proof is identical, so we omit it.
\begin{lemma}\label{lm: stackedness}
	Let $B_1$ and $B_2$ be PL balls of dimension $d_1$ and $d_2$, respectively. If $B_1$ is $i_1$-stacked and $B_2$ is $i_2$-stacked, then
	\begin{enumerate}
		\item The complex $B_1*B_2$ is an $(i_1+i_2)$-stacked PL $(d_1+d_2+1)$-ball. 
		\item Furthermore, if $d_1=d_2=d$ and  $B_1\cap B_2\subseteq \partial B_1\cap \partial B_2$ is a PL $(d-1)$-ball that is $i_3$-stacked, then $B_1\cup B_2$ is an $i$-stacked PL $d$-ball, where $i=\max\{i_1, i_2, i_3+1\}$.
	\end{enumerate}
\end{lemma}

We close this subsection with the following theorem that summarizes a few properties of the $h$-vectors of PL balls.

\begin{theorem} \label{h-vectors-properties}
Let $\Delta$ be a PL $(d-1)$-ball with $n$ vertices. Then
\begin{enumerate}
\item The $h$-numbers of $\Delta$ satisfy $0\leq h_i \leq \binom{n-d+i-1}{i}$ for all $1\leq i \leq d$.
\item $\Delta$ is $i_0$-neighborly if and only if $h_{i}(\Delta)=\binom{n-d+i-1}{i}$ for all $i\leq i_0$.
\item $\Delta$ is $(r-1)$-stacked if and only if $h_i(\Delta)=0$ for all $i\geq r$.
\end{enumerate}
\end{theorem}
The first two statements are due to Stanley \cite{Stanley75}, see also \cite[Chapter II.3]{Stanley96}; they hold not only for balls but for all Cohen--Macaulay complexes. For the last statement, see  \cite[Proposition 2.4]{Mc-04}.

\subsection{Kalai's squeezed balls and spheres}
To review the definition of squeezed balls and spheres, we will use some terminology from partially ordered sets. Specifically, recall that an {\em antichain} $\A$ in a poset $(\Qq, \leq)$ is a subset of $\Qq$ no two of which elements are comparable to each other. If $\A \subseteq \Qq$ is an antichain, we denote by $\Qq(\A)$ the \emph{order ideal} (also known as the \emph{initial set}) generated by $\A$: $\Qq(\A)=\{x\in \Qq: x\leq a \; \text{for some}\; a\in \A\}$. When $\Qq$ is finite, there is a natural bijection $\phi$ from the set of antichains of $\Qq$ to the set of order ideals of $\Qq$  defined by $\phi(\A)=\Qq(\A)$; the inverse map $\phi^{-1}$ takes an order ideal $\I$ to the set of maximal elements of $\I$.

In what follows, every $d$-subset of $[n]$ is written in an increasing order and is identified with an element of $\N^d$. In particular, we compare two $d$-subsets using the standard \emph{partial order} $\leq_p$ on $\N^d$:  for $F=\{i_1, i_2,\dots, i_d\}$ and $G=\{j_1, j_2, \dots, j_d\}$, we say that $G\leq_p F$ if $j_\ell\leq i_\ell$ for all $1\leq \ell\leq d$. Denote by $\mathbf{1}_{d}$ the all-ones vector of length $d$. We also say that $G\prec_p F$ if $G\leq_p F-\mathbf{1}_d$. 

Let $[m,n]$ denote the set $\{m, m+1,\ldots,n\}$. (It is the empty set if $m>n$.) Our main construction relies on the poset $\F_{2k}^{[m, n]}$ defined as follows.  When $k=0$, this poset consists only of the empty set. When $k\geq 1$, as a set  $\F_{2k}^{[m, n]}$ consists of the following facets of the cyclic polytope $C_{2k}(n)$:
$$\{\{i_1, i_1+1, i_2, i_2+1, \dots, i_k, i_k+1\}:\; m\leq i_1, \;i_k\leq n-1, \;i_j\leq i_{j+1}-2, \; \forall\; 1\leq j\leq k-1\};$$ these facets are ordered by the partial order $\leq_p$. Note that if $m>n-2k+1$, then $\F_{2k}^{[m, n]}$ is a void poset. Otherwise, $[n-2k+1, n]=\{n-2k+1,n-2k+2,\ldots,n\}$ is the unique maximal element of $\F_{2k}^{[m, n]}$, that is, $\F_{2k}^{[m, n]}$ is the order ideal generated by the antichain $\{[n-2k+1,n]\}$.  

When $k$ and $n$ are fixed or understood from context, we abbreviate $\F_{2k}^{[1, n]}$ as $\F_{2k}$ or as $\F$. We say that two antichains $\Ss$ and $\T$ in $\F$ satisfy $\T\leq_p \Ss$ if $\F(\T)\subseteq \F(\Ss)$; equivalently, if for every $G\in\T$ there is an element $F\in\Ss$ such that $G\leq_p F$. Similarly, we say that
$\T\prec_p \Ss$ if for every element $G\in \T$, there exists an element $F\in \Ss$ such that $G\prec_p F$. For instance, if $k=2$, $n=8$, $\Ss=\{\{1,2,7,8\}, \{3,4,6,7\}\}$, $\T=\{\{2,3,5,6\}\}$, and $\T'=\{\{1,2,6,7\}\}$, then $\T\prec_p \Ss$, $\T' \leq_p \Ss$, but $\T' \not\prec_p \Ss$.

For an antichain $\Ss$ in $\F$, let $B(\Ss)$ be the pure simplicial complex whose facets are the sets in the order ideal $\F(\Ss)$. (In particular, $B(\Ss)\neq B(\T)$ if $\Ss\neq \T$.) For example, if $k=2$ and $\Ss=\{\{1,2,5,6\}, \{2,3,4,5\}\}$, then the complex $B(\Ss)$ is a 3-ball with facets $\{1,2,3,4\}$, $\{1,2,4,5\}$, $\{1,2,5,6\}$ and $\{2,3,4,5\}$.

Kalai \cite{Kal} proved that the complexes $B(\Ss)$ are PL balls and called them \emph{squeezed balls}. The boundary complex $\partial B(\Ss)$ of $B(\Ss)$ is a \emph{squeezed sphere}. Some of the properties of these objects are summarized in the following theorem. We refer to \cite{Ziegler} for the definition of shellability and only mention a known fact that shellable balls and spheres are always PL.

\begin{theorem}\label{lm: prop of squeezed balls}
	Fix $k$ and $n$, and let $\Ss, \Ss'$ be non-empty antichains in $\F=\F_{2k}^{[1,n]}$. Then
	\begin{enumerate}
		\item $B(\Ss)$ is a $k$-stacked shellable $(2k-1)$-ball. Furthermore, if $\Ss=\{[n-2k+1, n]\}$, then $B(\Ss)$ is $k$-neighborly w.r.t.~$[n]$.
		\item If $\partial B(\Ss)=\partial B(\Ss')$, then $B(\Ss)=B(\Ss')$.
	\end{enumerate}
	\end{theorem}
\noindent Part 2 is \cite[Proposition 3.3]{Kal}, and a large portion of Part 1 is proved in \cite[Corollary 3.2 and Proposition 5.3(i)]{Kal}. For completeness, we discuss some of the details of the proof of Part 1 below.

\begin{proof}
By \cite[Corollary 3.2]{Kal}, $B(\Ss)$ is a shellable $(2k-1)$-ball. By Gale's evenness condition, the elements of $\mathcal{F}=\F_{2k}(\{[n-2k+1,n]\})$ are precisely the facets of $C_{2k}(n+1)$ that do not contain $n+1$. Hence $B(\{[n-2k+1,n]\})$ is $k$-neighborly w.r.t.~$[n]$ (because $C_{2k}(n+1)$ is $k$-neighborly w.r.t.~$[n+1]$) and
$$B(\Ss)\subseteq B(\{[n-2k+1, n]\})=\antst(n+1, \partial C_{2k}(n+1)).$$ Gale's evenness condition also implies that $\lk(n+1, \partial C_{2k}(n+1))= \partial C_{2k-1}(n)$. Consequently, $\lk(n+1, \partial C_{2k}(n+1))$ is a $(k-1)$-neighborly (w.r.t.~$[n]$) $(2k-2)$-sphere. Since this sphere is the boundary complex of $\antst(n+1, \partial C_{2k}(n+1))$, all faces of $\antst(n+1, \partial C_{2k}(n+1))$ of dimension $\leq k-2$ are boundary faces. We conclude that $\antst(n+1, \partial C_{2k}(n+1))$ is a $k$-neighborly w.r.t.~$[n]$ and $k$-stacked $(2k-1)$-ball. Finally, since $B(\Ss)$ is a full-dimensional subcomplex of this ball, an interior face of $B(\Ss)$ is necessarily an interior face of the antistar.  Thus $B(\Ss)$ is also $k$-stacked. 
\end{proof}

To count the number of distinct squeezed $(2k-2)$-spheres, we define another poset $$\PP=\PP_k^n=\{(x_1, x_2, \dots, x_k): 1\leq x_1< x_2< \dots <x_k\leq n-k\}\subseteq \N^k$$
also ordered by the partial order $\leq _p$. There is a natural bijection $R$ between $\PP$ and $\F$ given by 
$$ R: (x_1, x_2, \dots, x_k)\mapsto \{x_1, x_1+1, x_2+1, x_2+2, x_3+2, x_3+3,\dots, x_k+k-1, x_k+k\}.$$
This map is an isomorphism of posets. Counting the number of distinct antichains in $\PP$ leads to
\begin{theorem} {\rm \cite[Theorem 4.2]{Kal}}\label{thm: counting squeezed balls}
	Let $k\geq 2$. The number of combinatorial types of squeezed $(2k-1)$-balls with $n$ labeled vertices (or equivalently, those of squeezed $(2k-2)$-spheres with $n$ labeled vertices) is $2^{\Omega(n^{k-1})}$.
\end{theorem}

\section{The relative squeezed balls and spheres}
In this section, we introduce and study the main objects of the paper --- relative squeezed balls. For an antichain $\Ss$ in $\F=\F_{2k}^{[1,n]}$, let
\[
\Ss-\mathbf{1}_{2k}:=\{x-\mathbf{1}_{2k}: x=\{x_1, x_1+1, x_2, x_2+1, \dots, x_k, x_k+1\}\in \Ss, \;x_1>1\},
\]
and define $B_{\Ss}:=B(\Ss)\backslash B(\Ss - \mathbf{1}_{2k})$ to be the difference of two squeezed balls. The goal of this section is to prove the following result that parallels Theorem \ref{lm: prop of squeezed balls}:

\begin{theorem}\label{thm: prop of B_S}
	Let $\Ss, \Ss'$ be non-empty antichains in $\F=\F_{2k}^{[1,n]}$. Then
	\begin{enumerate}
		\item The complex $B_\Ss$ is a $(k-1)$-stacked PL $(2k-1)$-ball.
		Furthermore, if $\Ss$ contains $[1,2]\cup [n-2k+3, n]$ as an element, then $B_\Ss$ has $n$ vertices and is $(k-1)$-neighborly.
		\item If $\partial B_\Ss=\partial B_{\Ss'}$, then $B_\Ss=B_{\Ss'}$.
	\end{enumerate} 
\end{theorem}

In view of Theorem \ref{thm: prop of B_S}, we introduce the following terminology:

\begin{definition}
	Let $\Ss$ be a non-empty antichain in $\F=\F_{2k}^{[1,n]}$. The complex $B_\Ss$ is  called a \emph{relative squeezed $(2k-1)$-ball} defined by $\Ss$. The boundary complex $\partial B_\Ss$ is the \emph{relative squeezed $(2k-2)$-sphere} defined by $\Ss$. 
\end{definition}

To motivate this definition and Theorem \ref{thm: prop of B_S}, consider the following example: let $k=2$, $n=8$, and $\Ss=\{\{1,2,7,8\}, \{3,4,6,7\}\}$. Then $B(\Ss)$ is not $1$-stacked as it has too many facets (ten facets instead of five). On the other hand, $\Ss-\mathbf{1}_4=\{\{2,3,5,6\}\}$ and the facets of $B_{\Ss}$ consist of \[\{1,2,7,8\}, \{1,2,6,7\}, \{2,3,6,7\}, \{3,4,6,7\}, \,\text{and}\,\{3,4,5,6\}.\] The above order of facets of $B_{\Ss}$ shows that $B_{\Ss}$ is indeed a $1$-stacked $1$-neighborly (w.r.t.~$[8]$) $3$-ball.
		
The proof of Theorem \ref{thm: prop of B_S} requires quite a bit of preparation. To verify that $B_{\Ss}$ is a $(k-1)$-stacked PL ball, we will utilize Lemma \ref{lm: stackedness} along with inductive arguments on dimension. 
To this end, the elements of the form $[i, i']\cup [j, j']$ in an antichain of $\F^{[i, j']}$ will play a special role (e.g., notice the element $[1,2]\cup [n-2k+3,n]$ in the statement of Theorem \ref{thm: prop of B_S}) and the following definition will be indispensable.

\begin{definition}\label{main-def} Let $\Ss$ be an antichain in $\mathcal{F}=\F_{2k}^{[1,n]}$, let $1\leq \ell \leq k$, let $J=[j, j+2\ell-1]$ be a subset of $[n]$ of size $2\ell$, and let $m\geq 1$.
	\begin{itemize}
	\item Consider the following subcollection of $\F_{2(k-\ell)}^{[1, n]}$:
\[
\{H\subseteq [j+2\ell, n]:  J\cup H\in \F(\Ss)\}.
\]
Define $\Ss(J) \subseteq \F_{2(k-\ell)}^{[1, n]}$ to be the set of maximal elements of this collection. 

\item Define $\F(\Ss,m)=\F(\Ss)\cap \F^{[m,n]}$; that is, $\F(\Ss,m)$  is the collection of all sets $G$ in $\F(\Ss)$ such that the minimum of $G$ is at least $m$. Similarly, define $\F_{2(k-\ell)}(\Ss(J),m)$ as $\F_{2(k-\ell)}(\Ss(J)) \cap \F_{2(k-\ell)}^{[m,n]}$, where as before $\F_{2(k-\ell)}(\Ss(J))$ denotes the order ideal of $\F_{2(k-\ell)}^{[1,n]}$ generated by $\Ss(J)$.

\item 
Let $B(\Ss, m)$ be the pure simplicial complex whose collection of facets is $\F(\Ss,m)$. Similarly, let $B(\Ss(J),m)$ be the pure simplicial complex whose collection of facets is  $\F_{2(k-\ell)}(\Ss(J), m)$. In particular, $B(\Ss,1)=B(\Ss)$ and $B(\Ss(J),1)=B(\Ss(J))$.
\end{itemize}
\end{definition}

\begin{example}
	Consider the following antichain in $\F=\F_6^{[1,14]}$: $$\Ss:=\{\{1,2,3,4,13,14\}, \{1,2,6,7,11,12\}, \{2,3,4,5,12,13\}, \{2,3,5,6,10,11\},\{2,3,7,8,9,10\}\}.$$
	By definition,  $$\Ss([1,2])=\{\{3,4,13,14\},\{6,7,11,12\},\{4,5,12,13\},\{7,8,9,10\}\},$$ $$\Ss([2,3])=\{ \{4,5,12,13\}, \{5,6,10,11\}, \{7,8,9,10\}\}, \quad\mathrm{and}\quad \Ss([3,4])=\emptyset.$$
	Furthermore, $B(\Ss([2,3]), 6)$ is a 2-stacked 3-ball generated by the facets $$\{6,7,8,9\}, \{6,7,9,10\}, \{7,8,9,10\}.$$
    Note also that $\Ss([2,7])=\{\emptyset\}$ since $[2,7] \in \F(\Ss)$, but $\Ss([3,8])=\emptyset$ since $[3,8]\notin \F(\Ss)$.
\end{example}
To study complexes of the form $B(\Ss)\backslash B(\Ss-\mathbf{1}_{2k})$,
 it will be helpful to look at complexes of the form $B(\Ss) \backslash B(\T)$ for {\em all} pairs of antichains $\T \prec_p \Ss$. The following lemma, which is an easy consequence of Definition  \ref{main-def}, is the first step in this direction. For the rest of the section, we fix $k\geq 1$, $n\geq 2k$, and we always assume that $\Ss, \T$ are antichains in $\F=\F_{2k}^{[1,n]}$. 

\begin{lemma}\label{lm: basic prop}
	For $i, \ell \geq 1$,
	\begin{enumerate}
		\item $\Ss([i+1,i+2])\leq_p \Ss([i,i+1])$.
		\item $(\Ss-\mathbf{1}_{2k})([i,i+1])=\Ss([i+1,i+2])-\mathbf{1}_{2k-2}$.	
	  	\item If $\T\prec_p \Ss$, then $\T([i, i+2\ell-1])\prec_p \Ss([i+1, i+2\ell])$.
	\end{enumerate}
\end{lemma}
\begin{proof}
	 Parts 1 and 2 follow from Definition \ref{main-def}. 
	Part 3 is a consequence of Part 2. Indeed,  $$\T([i,i+1])\leq_p (\Ss-\mathbf{1}_{2k})([i,i+1])=\Ss([i+1,i+2])-\mathbf{1}_{2k-2}\prec_p \Ss([i+1,i+2]).$$
	Hence, for $\ell=2$,
	\[\T([i, i+3])=(\T([i, i+1]))([i+2, i+3]) \prec_p (\Ss([i+1, i+2]))([i+3, i+4])=\Ss([i+1, i+4]). \]
	The result now follows by induction on $\ell$.
\end{proof}

As the proof of Theorem \ref{thm: prop of B_S} is rather long and technical, it is worth to pause and outline the plan for the proof. Fix $\T\prec_p\Ss$. The first step is to decompose each complex $B(\Ss,i)\backslash B(\T,i)$ into analogous lower-dimensional objects joined with simplices, see Lemma \ref{B(S,i)-decomposition} and Corollary \ref{cor1}. The minimum of each facet $F$ determines which component of this decomposition $F$ is placed in. The second step is to study the intersections of components appearing in this decomposition, see Lemma \ref{lm: D_j cap D_{j+1}}. With these results at our disposal, the last step is to use induction on the dimension to show that each complex $B(\Ss,i)\backslash B(\T,i)$ is a ball with the desired properties, see Lemmas \ref{lm: B_S} and \ref{lem:B_S-neighb}.

\begin{lemma} \label{B(S,i)-decomposition}
The following decomposition holds:
	\begin{eqnarray}
	B(\Ss, i) &=&\bigcup_{j\geq i} \left( B\big(\Ss([j,j+1]), j+2\big)*\overline{[j, j+1]}\right)  \label{eq:line1}\\
	&=& \left(B\big(\Ss([i,i+1]), i+2\big)*\overline{[i, i+1]}\right) \cup B(\Ss, i+1) \nonumber.
	\end{eqnarray}
\end{lemma}
\begin{proof}
	By definition of $\F_{2(k-1)}(\Ss([j,j+1]), j+2)$, the facets of the complex on the right-hand side of \eqref{eq:line1} are also the facets of the complex on the left-hand side of \eqref{eq:line1}. Conversely, let $G$ be a facet of $B(\Ss, i)$ and let $j$ be the minimal element of $F$. Then $[j,j+1]\subseteq G$ and by definition of $\F(\Ss([j,j+1]))$, $G\backslash\{j, j+1\}\in \F_{2(k-1)}(\Ss([j,j+1]), j+2)$. Thus $G$ is also a facet of the complex on the right-hand side of \eqref{eq:line1}. 
\end{proof}
\begin{corollary}\label{cor1}
	If $\T\prec_p \Ss$, then
	$$B(\Ss, i)\backslash B(\T, i) =\bigcup_{j\geq i} \left(\big( B(\Ss([j,j+1]), j+2)\backslash B(\T([j,j+1]), j+2)\big)*\overline{[j, j+1]}\right).$$
\end{corollary}

 Our next step is to understand the intersections of components of the decomposition provided by Corollary \ref{cor1}. With this goal in mind, we fix $\Ss$ and $\T$ such that $\T\prec_p \Ss$ and introduce the following definition:
\begin{definition} \label{def:D_j}
Define $D_j$ as  
\[D_j:=\Big( B\big(\Ss([j,j+1]), j+2\big)\backslash B\big(\T([j,j+1]), j+2\big)\Big)*\overline{[j, j+1]}. \] 
In plain English, the complex $D_j$ is generated by all facets of $\F(\Ss)$ that are of the form $[j,j+1]\cup H$ with $H\subseteq [j+2,n]$, and are not facets of $\F(\T)$. Define also $\Gamma_{j,\ell}$ as
$$\Gamma_{j, \ell}:=B(\Ss([j+1,j+2\ell]), j+2\ell+1)\backslash B(\T([j, j+2\ell-1]), j+2\ell+1).$$
That is, the complex $\Gamma_{j, \ell}$ is generated by the facets $H\subseteq [j+2\ell+1, n]$ such that $[j+1, j+2\ell]\cup H\in \F(\Ss)$ but $[j, j+2\ell-1]\cup H\notin \F(\T)$.
\end{definition}

\begin{lemma}  \label{lemma:inters-of-D} If $D_{j+1}$ is not the void complex, then $D_j\cap D_{j+1}$ has the following decomposition according to initial segments of facets:
	\[D_j\cap D_{j+1}=\bigcup_{\ell=1}^{k}\left( \Gamma_{j, \ell}*\overline{[j+1, j+2\ell-1]}\right).\]
\end{lemma}

\begin{proof}
By definition, $D_j$ and $D_{j+1}$ are pure $(2k-1)$-dimensional simplicial complexes that do not share common facets. We first show that $D_j\cap D_{j+1}$ is pure $(2k-2)$-dimensional. Let $F$ be a maximal (w.r.t.~inclusion) face of $D_j\cap D_{j+1}$. Let $G$ be a minimal (w.r.t.~$\leq_p$) facet of $D_j$ containing $F$.
Note that $G$ must contain $[j,j+1]$. In addition, since $G\in\F_{2k}$, $G$ is a disjoint union of $k$ pairs of the form $[q, q+1]$. This implies that $G=[j, j+2\ell-1]\cup M$, where $1\leq\ell\leq k$ and $M\in \F_{2(k-\ell)}^{[j+2\ell+1,n]}$. It suffices to show that $\bar{G}:=G\backslash \{j\}\cup \{j+2\ell\}\in D_{j+1}$ and hence $F=G\backslash \{j\} =\bar{G}\backslash\{j+2\ell\}$. 

Suppose, to the contrary, that $\bar{G}\notin D_{j+1}$. Since $G\leq_p \bar{G}$ and $G\notin B(\T)$, it follows that $\bar{G}\notin B(\T)$, and so $\bar{G}\notin B(\Ss)$ (or else, $\bar{G}$ would be in $D_{j+1}$). The fact that $G\in B(\Ss)$ then forces $G$ to be in $\Ss$. (Indeed, if $G$ is in $B(\Ss)$ but not in $\Ss$, then there must  exist  $G'\in\Ss$ such that $G<_p G'$. By definition of $\bar{G}$, such $G'\in B(\Ss)$ satisfies $\bar{G}\leq_p G'$, which is impossible because $\bar{G}\notin B(\Ss)$.)

Let $s:= \max(G\backslash F)$.  If $s\in[j,j+2\ell-1]$, then $M\subseteq F$, which together with $F\in D_{j+1}$ implies that  $\bar{G}=[j+1, j+2\ell] \cup M \in D_{j+1}$, contradicting our assumption. If $s\notin[j,j+2\ell-1]$, we let $H$ be the maximum (w.r.t.~$\leq_p$) facet of $\F_{2k}$ such that $H<_p G$ and $G\backslash\{s\}\subseteq H$. Our assumption that  $j+2\ell\notin G$ and the  definition of $\F_{2k}$ imply that $H$ exists and that it can be expressed as $H=[j, j+2\ell-1]\cup (M\backslash \{s\}) \cup \{t\}$ for some $t$ between $j+2\ell$ and $s-1$. Hence, $\min(H)=j$ and $H<_p G<_p H+\mathbf{1}_{2k}$. The fact that $G\in\Ss$, then implies that $H\in B(\Ss)\backslash B(\Ss-\mathbf{1}_{2k}) \subseteq B(\Ss)\backslash B(\T)$. Thus, $H\in D_j$. This, however, contradicts our choice of $G$ as a minimal facet of $D_j$ containing $F$.

The above discussion shows that any maximal face $F\in D_j\cap D_{j+1}$ is a $(2k-2)$-face with the property that for some $1\leq\ell\leq k$ , 
 $$[j+1,j+2\ell-1] \subseteq F, \quad j\notin F, \quad j+2\ell\notin F,$$
and furthermore $F\cup\{j\}$ is a facet of $D_j$ while $F\cup\{j+2\ell\}$ is a facet of $D_{j+1}$. We conclude that  $F\backslash [j+1, j+2\ell-1]$ is a common facet of complexes
\begin{eqnarray*}
&B(\Ss([j, j+2\ell-1]), j+2\ell+1)\backslash B(\T([j, j+2\ell-1]), j+2\ell+1)& \, \mbox{ and}\\
&B(\Ss([j+1, j+2\ell]), j+2\ell+1)\backslash B(\T([j+1, j+2\ell]), j+2\ell+1).&
\end{eqnarray*}
Since by Part 3 of Lemma~\ref{lm: basic prop}, $B(\T([j,j+2\ell-1]),j+2\ell+1) \prec_p B(\Ss([j+1,j+2\ell]),j+2\ell+1)$, all common facets of the above two complexes, including $F\backslash [j+1,j+2\ell-1]$, are facets of $$B(\Ss([j+1, j+2\ell]), j+2\ell+1)\backslash B(\T([j, j+2\ell-1]), j+2\ell+1)=\Gamma_{j, \ell}.$$
	We infer that $D_j\cap D_{j+1}\subseteq \bigcup_{\ell=1}^{k}\left( \Gamma_{j, \ell}*\overline{[j+1, j+2\ell-1]}\right).$

	For the other inclusion, assume that $H$ is a facet of $\Gamma_{j, \ell}$. We need to show that $$
		[j, j+2\ell-1] \cup H\in D_j \quad \mbox{and} \quad [j+1, j+2\ell]\cup H\in D_{j+1},$$ or, equivalently, that
\begin{eqnarray*}
&[j+2, j+2\ell-1]\cup H\in B(\Ss([j,j+1]), j+2)\backslash B(\T([j,j+1]), j+2)& \, \mbox{ and}\\ &[j+3, j+2\ell]\cup H\in B(\Ss([j+1,j+2]), j+3)\backslash B(\T([j+1,j+2]), j+3).&
\end{eqnarray*}
This follows easily from our assumption that $H\in\Gamma_{j,\ell}$ using the definition of $\Ss([r, r+2\ell-1])$.
\end{proof}

We are now ready to present a much more elegant description of $D_j\cap D_{j+1}$.
\begin{lemma}\label{lm: D_j cap D_{j+1}} If $D_{j+1}$ is not the void complex, then
	$$D_j\cap D_{j+1}=\Big( B(\Ss([j+1,j+2]), j+2)\backslash B(\T([j,j+1]), j+2)\Big) *\overline{\{j+1\}}.$$
\end{lemma}
\begin{proof}
	It suffices to prove that
	\begin{equation*}
		\begin{split} &\quad \bigcup_{\ell=m}^k\left( \Gamma_{j, \ell}*\overline{[j+1, j+2\ell-1]}\right)\\&=\Big( B(\Ss([j+1, j+2m]), j+2m)\backslash B(\T([j, j+2m-1]), j+2m)\Big)*\overline{[j+1, j+2m-1]}.
		\end{split}
	\end{equation*}
	Indeed, the case $m=1$ of this equation together with Lemma \ref{lemma:inters-of-D} immediately yield the statement.
	
	The proof is by reverse induction on $m$. The base case $m=k$ follows from the fact that $\Ss([j+1,j+2k])=\{\emptyset\}$, while $\T([j,j+2k-1])$ is either $\{\emptyset\}$ or $\emptyset$. In any case, 
	\begin{eqnarray*}
	&B(\Ss([j+1,j+2k]), j+2k+1)= B(\Ss([j+1,j+2k]), j+2k)& \, \mbox{ and} \\
	&B(\T([j,j+2k-1]), j+2k+1)=(\T([j,j+2k-1]), j+2k).
	\end{eqnarray*}

The inductive step is a consequence of the following two observations. (Recall that $\overline{\emptyset}=\{\emptyset\}$.)
	\begin{equation*}
		\begin{split}
			&\quad B(\Ss([j+1, j+2m]), j+2m)\backslash B(\T([j, j+2m-1]), j+2m)\\
			&\stackrel{(*)}{=}\left( \big(B(\Ss([j+1, j+2m+2]), j+2m+2)\backslash B(\T([j,  j+2m+1]), j+2m+2)\big) *\overline{[j+2m, j+2m+1]}\right) \\
			&\quad \cup \Big( B(\Ss([j+1, j+2m]), j+2m+1)\backslash B(\T([j,   j+2m-1]), j+2m+1)\Big)\\
			&\stackrel{(**)}{=}\bigcup_{\ell=m}^k  \left( \Gamma_{j, \ell}*\overline{[j+2m, j+2\ell-1]}\right).
		\end{split}
	\end{equation*}
Here $(**)$ follows from the inductive hypothesis along with the definition of $\Gamma_{j, \ell}$. For $(*)$, note that $H=[j+2m, j+2m+1]\cup G$ is a facet of $B(\Ss([j+1, j+2m]), j+2m)$ if and only if the minimum element of $G$ is at least $j+2m+2$, and furthermore there exists a minimal (w.r.t. $\leq_p$) facet $H'$ such that $H\leq_p H'$ and $[j+1, j+2m]\cup H'\in B(\Ss)$. This $H'$ must be of the form $[j+2m+1, j+2m+2]\cup G'$ for some $G'\geq_p G$. Hence $G'\in B(\Ss([j+1, j+2m+2]), j+2m+3)$ and $G\in B(\Ss([j+1, j+2m+2]), j+2m+2)$. Thus by Lemma \ref{B(S,i)-decomposition}, 
\begin{equation*}
	\begin{split}
		B(\Ss([j+1, j+2m]), j+2m)&=\Big(B(\Ss([j+1, j+2m+2]), j+2m+2) *\overline{[j+2m, j+2m+1]}\Big) \\
		&\quad \cup B(\Ss([j+1, j+2m]), j+2m+1).
	\end{split}
\end{equation*}
This expression, along with the expression for $B(\T([j, j+2m-1]), j+2m)$ given by Lemma \ref{B(S,i)-decomposition}, establishes $(*)$ and completes the proof of the lemma.
\end{proof}

With Corollary \ref{cor1} and Lemma \ref{lm: D_j cap D_{j+1}} at our disposal, we are now in a position to prove the portion of Theorem \ref{thm: prop of B_S} asserting that $B_{\Ss}$ is a $(k-1)$-stacked PL $(2k-1)$-ball. In fact, we prove the following stronger statement.

\begin{lemma}\label{lm: B_S}
If $\T\prec_p \Ss$ and $B(\Ss,i)$ is not the void complex, then $B(\Ss, i)\backslash B(\T, i)$ is a $k$-stacked PL $(2k-1)$-ball. Furthermore, it is $(k-1)$-stacked if $\T=\Ss-\mathbf{1}_{2k}$.
\end{lemma}
\begin{proof}
	The proof is by induction on $k$. For $k=1$, there exist integers $i<j<j'$ such that $B(\Ss, i)$ is generated by edges $\{i,i+1\},\{i+1, i+2\}, \ldots, \{j'-1,j'\}$ while $B(\T,i)$ is generated by edges $\{i,i+1\}, \{i+1,i+2\}, \ldots, \{j-1,j\}$. Hence $B(\Ss, i)\backslash B(\T, i)$ is a path, and so it is indeed a $1$-stacked $1$-ball. If $j=j'-1$ or, equivalently, if $T=S-\mathbf{1}_2$, this path consists of a single edge, and hence it is $0$-stacked.

    For the inductive step, note that by Corollary \ref{cor1}, $B(\Ss, i)\backslash B(\T, i)=\cup_{j\geq i} D_j$. By definition of $D_j$ (see Definition \ref{def:D_j}), Lemma \ref{lm: D_j cap D_{j+1}}, the inductive hypothesis, and Part 1 of Lemma \ref{lm: stackedness}, each $D_j$ is a $(k-1)$-stacked PL $(2k-1)$-ball, while each $D_j\cap D_{j+1}$ is a $(k-1)$-stacked PL $(2k-2)$-ball. Since $\Ss([j+1,j+2])\leq_p \Ss([j,j+1])$ and $\T([j+1,j+2])\leq_p \T([j,j+1])$ for all $j$, it follows 
	from the definition of $D_j$ 
	that $$D_{j+1}\cap D_{i}\subseteq D_{j+1}\cap D_{i+1}\subseteq \dots \subseteq D_{j+1}\cap D_j.$$ Hence $D_{j+1}\cap (D_i\cup D_{i+1}\cup \dots \cup D_{j})=D_j\cap D_{j+1}$. By induction on $\ell-i$ and using Part 2 of Lemma \ref{lm: stackedness}, we infer that $\cup_{i\leq j\leq \ell} D_j$ is a $k$-stacked PL $(2k-1)$-ball. This proves the first claim.
    
    Finally, if $\T=\Ss-\mathbf{1}_{2k}$, then, by inductive assumptions, every intersection $D_j\cap D_{j+1}$ is $(k-2)$-stacked. Hence the same argument as above using Part 2 of Lemma 2.2 yields that $B(\Ss, i)\backslash B(\T, i)$ is $(k-1)$-stacked.
\end{proof}

\begin{remark}
		The proof of Lemma \ref{lm: B_S} also implies that the ball $B(\Ss, i)\backslash B(\T, i)$ is constructible, see \cite[Section 11]{Bjorner95} for the definition and properties of constructibility.
\end{remark}

The part of Theorem \ref{thm: prop of B_S} asserting that if $\partial B_{\Ss}=\partial B_{\Ss'}$, then  $B_{\Ss}=B_{\Ss'}$ is now immediate.

\smallskip\noindent {\it Proof of Theorem \ref{thm: prop of B_S}(2): \ }
Consider a PL $(2k-2)$-sphere  $K:=\partial B_{\Ss}=\partial B_{\Ss'}$. By a result due to McMullen \cite[Theorem 3.3]{Mc-04} (for polytopal spheres) and due to Bagchi and Datta \cite[Theorem 2.12]{BD} (for triangulated spheres), a $(k-1)$-stacked PL ball $B$ that satisfies $\partial B=K$ is unique. The result follows since by Lemma \ref{lm: B_S} both  $B_{\Ss}$ and $B_{\Ss'}$ are $(k-1)$-stacked PL balls.
\hfill$\square$\medskip

To complete the proof of Theorem \ref{thm: prop of B_S}, it only remains to show that if $\Ss$ contains the set $[1,2] \cup [n-2k+3, n]$, then $B_\Ss$ is $(k-1)$-neighborly w.r.t.~$[n]$. To do so, we first count the number of facets of such $B_{\Ss}$. For this part of the proof, it is more convenient to work with the poset $\PP=\PP_k^n$ (introduced at the end of Section 2.3) instead of $\F=\F_{2k}^{[1,n]}$. For an antichain $\A$ of $\mathcal{P}$, define $$\A-\mathbf{1}_k=\{x-\mathbf{1}_k: x=(x_1, x_2, \dots, x_k)\in \A, \;x_1>1\} \quad \mbox{and} \quad \PP_{\A}=\PP(\A) \backslash \PP(\A-\mathbf{1}_k).$$ Note that the isomorphism $R: \PP\to \F$ commutes with subtracting the all-ones vector: $R(\A-\mathbf{1}_{k})=R(\A)-\mathbf{1}_{2k}$ and $R^{-1}(\Ss-\mathbf{1}_{2k})=R^{-1}(\Ss)-\mathbf{1}_{k}$.

\begin{lemma}\label{lm: counting facets}
	Let $\A$ be an antichain of $\mathcal{P}$ that contains $G=(1, n-2k+2, \dots, n-k)$. Then $|\mathcal{P}_{\{G\}}|=|\mathcal{P}_{\A}|=\binom{n-k-1}{k-1}$. Equivalently, if $\Ss$ is an antichain of $\mathcal{F}$ that contains $R(G)=[1,2]\cup[n-2k+3,n]$, then the number of facets of $B_\Ss$ is $\binom{n-k-1}{k-1}$.
\end{lemma}
\begin{proof}
	First note that 
	$$\mathcal{P}_{\{G\}}=\mathcal{P}(\{G\})=\{(1, x_2, x_3, \dots, x_k): 1<x_2<x_3<\dots<x_k\leq n-k\}.$$
	Thus $|\mathcal{P}_{\{G\}}|=\binom{n-k-1}{k-1}$. 
	
	It remains to show that for a fixed antichain $\A$ that contains $G$, $|\mathcal{P}_{\{G\}}|=|\mathcal{P}_{\A}|$. This will be done once we show that the following map is a bijection:
	\begin{equation*}
		\begin{split}
			L: \mathcal{P}_\A &\to \mathcal{P}_{\{G\}}\\
			x=(x_1, \dots, x_k)&\mapsto x-(x_1-1)\cdot \mathbf{1}_{k}=(1, x_2-x_1+1, x_3-x_1+1, \dots, x_k-x_1+1).
		\end{split}
	\end{equation*}
	
	To see that $L$ is one-to-one, observe that for any $x, y\in \mathcal{P}$ such that $y=x+a\cdot \mathbf{1}_k$ for some $a\geq 1$, only $x$ or $y$ can be in $\mathcal{P}_\A$ but not both. Indeed, if $y\in \mathcal{P}_\A$, then $y\in \mathcal{P}(\A)$. Thus  $x\in \mathcal{P}(\A-\mathbf{1}_{k})$, and hence $x\notin \mathcal{P}_\A$. To see that $L$ is onto, note that any element of $\mathcal{P}$ that is of the form $(1, x_2, x_3, \dots, x_k)$ is in $\mathcal{P}_{\{G\}}\subseteq \mathcal{P}(\A)$. Consider the smallest $a\geq 1$ such that $(1, x_2, x_3, \dots, x_k)+a\cdot \mathbf{1}_k\notin \mathcal{P}(\A)$ (it exists since for $a$ sufficiently large, $(1, x_2, \dots, x_k)+ a\cdot \mathbf{1}_k$ is not even in $\mathcal{P}$). Then $(1, x_2, \dots, x_k)+(a-1)\cdot \mathbf{1}_k\in \mathcal{P}_\A$, and its image under $L$ is  $(1, x_2, \dots, x_k)$.
\end{proof}

The neighborliness of $B_\Ss$ now follows easily:
\begin{lemma} \label{lem:B_S-neighb}
Let $\Ss$ be an antichain in $\F$ that contains the set $[1,2] \cup [n-2k+3,n]$. Then $B_\Ss$ is $(k-1)$-neighborly w.r.t.~$[n]$.
\end{lemma}
\begin{proof}
By Lemma \ref{lm: counting facets}, $f_{2k-1}(B_\Ss)= \binom{n-k-1}{k-1}$. Also, by Lemma \ref{lm: B_S},  $B_\Ss$ is a $(k-1)$-stacked PL $(2k-1)$-ball, and so $h_i(B_\Ss)=0$ for all $k\leq i\leq 2k$ (see Theorem \ref{h-vectors-properties}(3)). Thus
\[
\binom{n-k-1}{k-1}=f_{2k-1}(B_\Ss)=\sum_{i=0}^{2k} h_i(B_\Ss)=\sum_{i=0}^{k-1}h_i(B_\Ss).
\]
Since $\binom{n-k-1}{k-1}=\sum_{i=0}^{k-1} \binom{n-2k+i-1}{i}$ and since by Theorem \ref{h-vectors-properties}(1), $h_i(B_\Ss)\leq \binom{n-2k+i-1}{i}$ for all $i$, it follows that 
$h_i(B_\Ss)=\binom{n-2k+i-1}{i}$ for all $i\leq k-1$, which in turn implies that $B_\Ss$ is $(k-1)$-neighborly w.r.t~$[n]$ (see Theorem \ref{h-vectors-properties}(2)).
\end{proof}

\noindent This concludes the proof of Theorem \ref{thm: prop of B_S}.

\section{The number of neighborly $(d-1)$-spheres on $n$ vertices}
In this section we prove our main result, Theorem \ref{main-thm}, asserting that $sn(d, n)\geq 2^{\Omega(n^{\lfloor (d-1)/2)\rfloor}}$. The following lemma provides an inductive method that given a neighborly sphere generates a new neighborly sphere with one additional vertex. This result is known and was used extensively to construct neighborly complexes. We refer to \cite{Padrol-13,Shemer} for a similar method (known as the sewing method) that was used to construct neighborly polytopes and neighborly oriented matroids; see also \cite[Lemma 3.1]{N-Z} for an analogous statement in the centrally symmetric case.

\begin{lemma}\label{lm: neighborly n+1 ver}
	Let $\Delta$ be a $\lfloor d/2\rfloor$-neighborly PL $(d-1)$-sphere on the vertex set $[n]$. Let $B$ be a $(\lfloor d/2\rfloor-1)$-neighborly (w.r.t.~$[n]$) and $(\lfloor d/2\rfloor-1)$-stacked PL $(d-1)$-ball contained in $\Delta$. Then the complex $\Delta(B)$ obtained from $\Delta$ by replacing $B$ with $\partial B*\overline{\{n+1\}}$ is a $\lfloor d/2\rfloor$-neighborly PL $(d-1)$-sphere on $[n+1]$.
\end{lemma}
\begin{proof}
	First note that $B$ and $\partial B* \overline{\{n+1\}}$ are PL $(d-1)$-balls with the same boundary. Hence $\Delta\backslash B$ is a PL $(d-1)$-ball and $\Delta(B)=(\Delta\backslash B)\cup (\partial B*\overline{\{n+1\}})$ is a PL $(d-1)$-sphere. Moreover, since $B$ is $(\lfloor d/2\rfloor-1)$-stacked, it follows that 
	\begin{eqnarray*}
	&\skel_{\lfloor d/2\rfloor -1}(\Delta\backslash B)=\skel_{\lfloor d/2\rfloor -1}(\Delta), \quad \mbox{ and }&\\
	&\skel_{\lfloor d/2\rfloor -2}(\lk(n+1, \Delta(B)))=\skel_{\lfloor d/2\rfloor -2}(\partial B)=\skel_{\lfloor d/2\rfloor -2}(B).&
	\end{eqnarray*}
The fact that $\Delta$ is $\lfloor d/2\rfloor$-neighborly and $B$ is $(\lfloor d/2\rfloor-1)$-neighborly w.r.t.~$[n]$ then shows that $\Delta(B)$ is $\lfloor d/2\rfloor$-neighborly w.r.t.~$[n+1]$.
\end{proof}
\begin{theorem}\label{thm: counting relative squeezed balls}
	Let $k\geq 2$. The number of distinct labeled $(k-1)$-neighborly (w.r.t.~$[n]$) and $(k-1)$-stacked PL $(2k-1)$-balls that are contained in $\partial C_{2k}(n)$ is at least $2^{\Omega(n^{k-1})}$.
\end{theorem}
\begin{proof}
	By Theorem \ref{thm: prop of B_S}, $B_\Ss$ is a $(k-1)$-neighborly (w.r.t.~$[n]$) and $(k-1)$-stacked PL $(2k-1)$-ball in $\partial C_{2k}(n)$ for each antichain $\Ss$ in $\mathcal{F}$ that contains the set $[1,2]\cup [n-2k+3,n]$. All these balls are distinct labeled balls since their sets of maximal facets (w.r.t.~$\leq_p$) are exactly the antichains $\Ss$. The number of such balls is the number of antichains containing $[1,2]\cup [n-2k+3,n]$, which is at least as large as the number of antichains in $\F_{2k}^{[3,n-2k+2]}$. As $\F_{2k}^{[3,n-2k+2]}$ is isomorphic to $\F_{2k}^{[1,n-2k]}$, the number of such antichains is at least $2^{\Omega((n-2k)^{k-1})}=2^{\Omega(n^{k-1})}$ by Theorem \ref{thm: counting squeezed balls}.	
\end{proof}

We are finally ready to prove our main result, Theorem \ref{main-thm}, asserting that for $d \geq 5$, the number of combinatorial types of $\lfloor d/2\rfloor$-neighborly $(d-1)$-spheres on $n$ labeled vertices is at least $2^{\Omega(n^{\lfloor (d-1)/2\rfloor})}$.

\smallskip\noindent {\it Proof of Theorem \ref{main-thm}: \ }
Consider the family $\mathcal{H}$ of $(k-1)$-neighborly (w.r.t.~$[n]$) and $(k-1)$-stacked PL $(2k-1)$-balls contained in $\Delta:=\partial C_{2k}(n)$. In the case of $d=2k$, apply Lemma \ref{lm: neighborly n+1 ver} to $\Delta$ and a ball $B$ in this family to obtain the complex $\Delta(B)$ that is a $k$-neighborly PL $(2k-1)$-sphere on $[n+1]$. These spheres $\Delta(B)$ are pairwise distinct because their restrictions to the vertex set $[n]$ are exactly the complexes $\Delta\backslash B$,  and these are pairwise distinct. The result then follows from Theorem~\ref{thm: counting relative squeezed balls}.
	
	In the case of $d=2k-1$, consider the boundary complex of $B$ for each $B\in\mathcal{H}$. Since $B$ is a $(k-1)$-stacked PL $(2k-1)$-ball, all faces of $B$ of  dimension $\leq k-2$ are in $\partial B$. The fact that $B$ is $(k-1)$-neighborly then guarantees that $\partial B$ is a $(k-1)$-neighborly PL $(2k-2)$-sphere. Furthermore, since the boundary complex of a $(k-1)$-stacked PL $(2k-1)$-ball uniquely determines that ball \cite[Theorem 2.12]{BD}, distinct elements of $\mathcal{H}$  have distinct boundary complexes. The lower bound again follows from Theorem \ref{thm: counting relative squeezed balls}.
\hfill$\square$\medskip

\begin{remark}
     For $d\geq 5$, the number of combinatorial types of \emph{unlabeled} $\lfloor d/2\rfloor$-neighborly $(d-1)$-spheres on $n$ vertices is also at least $2^{\Omega(n^{\lfloor (d-1)/2\rfloor})}$. This is because dividing the lower bound by  $n! = 2^{O(n\log n)}$ does not affect its asymptotic growth if $d\geq 5$.
\end{remark}

We end the paper with an open problem. By the results of \cite{Kal} and \cite{Lee00}, both squeezed balls and squeezed spheres are shellable. It is natural to ask whether relative squeezed balls and spheres are also shellable. More generally, we pose the following problem.
\begin{question}
	Let $k\geq 1$. Let $\T\prec_p \Ss$ be non-empty antichains in $\mathcal{F}_{2k}^{[1,n]}$. Is $B(\Ss)\backslash B(\T)$ shellable? Is $\partial (B(\Ss)\backslash B(\T))$ shellable?
\end{question}
By Corollary \ref{cor1}, we write $B(\Ss)\backslash B(\T)$ as $\cup_{j=1}^\ell D_j$. In the first nontrivial case $k=2$, a shelling order for $B(\Ss)\backslash B(\T)$ can be given as follows: $$(F_{1,1}, \dots, F_{1, m_1}, F_{2, 1}, \dots, F_{2, m_2}, \dots, F_{\ell, 1}, \dots, F_{\ell, m_\ell}),$$
where $(F_{i, 1}, \dots, F_{i, m_i})$ is the unique shelling order of $D_i$ induced by the {\em reverse} partial order on the path $B(\Ss[i, i+1], i+2)\backslash B(\T([i, i+1]), i+2)$.

		



	\bibliography{refs}
	\bibliographystyle{plain}
\end{document}